\documentclass[a4paper, 11pt]{amsart}

\usepackage[latin1]{inputenc}
\usepackage{enumerate}
\usepackage{mathrsfs}
\usepackage{amssymb,amsmath}
\usepackage[all]{xy}

\title[]{Global Okounkov bodies for Bott-Samelson varieties}
\author{David Schmitz and Henrik Sepp\"anen} 

\thanks{The first author was supported by DFG grant BA 1559/6-1. The second author was supported by the DFG Priority Programme 1388 ``Representation Theory"}

\date{\today}

\address{David Schmitz,
   Fach\-be\-reich Ma\-the\-ma\-tik und In\-for\-ma\-tik,
   Philipps-Uni\-ver\-sit\"at Mar\-burg,
   Hans-Meer\-wein-Stra{\ss}e,
   D-35032~Mar\-burg, Germany.}
\email{schmitzd@mathematik.uni-marburg.de}

\address{Henrik Sepp\"{a}nen,
Mathematisches Institut,
Georg-August-Univer\-sit\"at G\"ot\-tingen,
Bunsenstra\ss e 3-5, 
D-37073 G\"ottingen,
Germany}
\email{hseppaen@uni-math.gwdg.de}

\newcommand{\C}{\mathbb{C}}
\newcommand{\R}{\mathbb{R}}
\newcommand{\N}{\mathbb{N}}
\newcommand{\Z}{\mathbb{Z}}
\newcommand{\Q}{\mathbb{Q}}
\newcommand{\mP}{\mathbb{P}}

\newcommand{\spec}{\mbox{Spec}}

\newcommand{\Nef}{\mbox{Nef}}
\newcommand{\Eff}{\overline{\mbox{Eff}}}
\newcommand{\exc}{\mbox{exc}}
\newcommand{\vol}{\mbox{vol}}
\newcommand{\Cox}{\mbox{Cox}}

\renewcommand{\to}{\longrightarrow}
\renewcommand{\O}{\mathcal O}

\newcommand\be{\begingroup\arraycolsep=0.13888em\begin{eqnarray*}}
\newcommand\ee{\end{eqnarray*}\endgroup}

\newtheorem{prop}{Proposition}[section]
\newtheorem{lemma}[prop]{Lemma}
\newtheorem{cor}[prop]{Corollary}
\newtheorem{thm}[prop]{Theorem}
\newtheorem{introthm}{Theorem}

\newtheorem{introcor}[introthm]{Corollary}
\theoremstyle{definition}
\newtheorem{defin}[prop]{Definition}
\newtheorem{rem}[prop]{Remark}

\begin{document}

\begin{abstract}
We use the theory of Mori dream spaces to prove that the 
global Okounkov body of a Bott-Samelson variety with respect to 
a natural flag of subvarieties is rational polyhedral. In fact, we prove more 
generally that this holds for any Mori dream space which admits a flag
of Mori dream spaces satisfying a certain regularity condition. As a corollary, 
Okounkov bodies of effective line bundles over Schubert varieties 
are shown to be rational polyhedral. In particular, it follows that the 
global Okounkov body of a flag variety $G/B$ is rational 
polyhedral. 

As an application we show that the asymptotic behaviour of
dimensions of weight spaces in section spaces of line bundles is given by the counting of 
lattice points in polytopes.
\end{abstract} 

\maketitle

\section*{Introduction}
Okounkov bodies were first introduced by A. Okounkov in his famous paper \cite{Ok96} as 
a tool for studying multiplicities of group representations. The idea is that 
one should be able to approximate these multiplicities by counting the number of integral points 
in a certain convex body in $\R^n$. More precisely, the setting is the 
following. Let $G$ be a complex reductive group which acts as automorphisms 
on an effective line bundle $L$ over a projective variety $X$, and hence defines a 
representation on the space of sections $H^0(X, L^k)$ for  each integral power, $L^k$, of $L$. 
Okounkov constructs a convex compact set $\Delta \subseteq \R^n$, where 
$n=\mbox{dim}\, X$, with the property that for each irreducible finite-dimensional representation 
$V_\lambda$, where $\lambda$---the so-called \emph{highest weight}---is 
a parameter, the multiplicity $m_{k\lambda, k}:=\mbox{dim} \,\, \mbox{Hom}_G(V_{k\lambda}, H^0(X, L^k))$
of $V_{k\lambda}$ in $H^0(X, L^k)$ is asymptotically given by the 
volume of the convex body $\Delta_\lambda:=\Delta \cap H_\lambda$, where $H_\lambda \subseteq \R^{n+1}$ 
is a certain affine subspace, in the following sense:
\begin{align}
\lim_{k \rightarrow \infty} \frac{m_{k\lambda, k}}{k^m}=\mbox{vol}_m(\Delta_\lambda), 
\label{E: approxmult}
\end{align}
where $m$ is the dimension of $\Delta_\lambda$, and the volume on the right hand side 
denotes the $m$-dimensional Euclidean volume of $\Delta_\lambda$. 
An approximation of the integral $\mbox{vol}_m(\Delta_\lambda)$ by Riemann sums 
yields that the multiplicity $m_{k\lambda, k}$ is asymptotically 
given by the number of points of the set $\Delta_\lambda \cap \frac{1}{k}\Z^m$. 

The construction of the body $\Delta$ is purely geometric and 
depends on a choice of a flag $Y_\bullet$, $Y_n \subseteq Y_{n-1} \subseteq \cdots \subseteq Y_0=X$ 
of irreducible subvarieties of 
$X$, and the ``successive orders of vanishing'' of certain invariant 
sections $s \in H^0(X, L^k)$ along this flag. It was later realized by Kaveh and Khovanskii (\cite{KK09}), 
and independently by Lazarsfeld and Musta{\c{t}}{\u{a}}, (\cite{LM09}), that Okounkov's construction 
makes sense for more general subseries of the section ring $R(X, L)$ of a line 
bundle over a variety $X$, and that the asymptotics of dimensions of linear series
can be studied by counting lattice points of convex compact bodies. For example, 
the analog of \eqref{E: approxmult} for the complete linear series of a big line bundle 
$L$ is given by the identity
\begin{align*}
\lim_{k \rightarrow \infty} \frac{h^0(X, L^k)}{n! k^n}=\frac{1}{n!}\mbox{vol}_n(\Delta_{Y_\bullet}(L)), 
\end{align*}
where $\Delta_{X_\bullet}(L)$ denotes the Okounkov body of the line bundle 
$L$ with respect to the flag $Y_\bullet$. 

The above formula shows in particular that the volume of the Okounkov body is 
an invariant of the line bundle $L$, and thus does not depend on the choice 
of the flag $Y_\bullet$. However, the shape of $\Delta_{X_\bullet}(L)$ depends heavily on the 
flag, and it is a notoriously hard problem to explicitly describe these 
bodies, or even to show that they possess some nice properties, such as 
being polyhedral. A yet more difficult problem is to determine the
global Okounkov body $\Delta_{Y_\bullet}(X)$ of a variety $X$ (cf. \cite{LM09}), which is a convex cone 
in a certain Euclidean space, and no longer depends on a particular line bundle $L$.

Returning to the original motivation by Okounkov of studying multiplicites 
of representations, there is also another approach to describing 
multiplicities by counting lattice points in convex bodies, namely Littelmann's 
construction of string polytopes (\cite{Li98}). The setting here is the following. 
Let $G$, again, be a complex reductive group, and let $H \subseteq G$ 
be a maximal torus in $G$. Then any irreducible finite-dimensional $G$-representation 
$V_\lambda$ admits a basis of weight vectors with respect to $H$, and this 
basis is parameterized by the integral points in a rational 
polytope $C^\lambda$, the \emph{string polytope} of $V_\lambda$. Moreover, the 
approximative lattice counting problem is even exact here. Since the 
irreducible representations $V_\lambda$ can be realized as section spaces 
$H^0(X, L_\lambda)$, where $X=G/B$ for a Borel subgroup $B \subseteq G$, and 
$L_\lambda$ is a line bundle over $X$, it would be interesting to 
recover Littelmann's string polytopes $C^\lambda$ as Okounkov bodies, 
or at least to construct rational polyhedral Okounkov bodies which describe 
asymptotic multiplicities of weight spaces.

In the present paper we study both problems described above---namely the 
Okounkov bodies for complete linear series, and the asymptotics of weight 
multiplicites---for general Bott-Samelson varieties $Z=Z_w$ (given 
by a reduced expression $w$ for an element $\overline{w}$ in the Weyl group of $G$), 
that is, Bott-Samelson varieties which desingularize some Schubert variety 
$X_w$ in a flag variety $G/B$. 

Our main result is formulated in the more general context of Mori dream spaces. 
It is well known that a Mori dream space $X$ admits finitely many small $\Q$-factorial
modifications such that any movable divisor $D$ on $X$ is the pullback of a nef divisor under 
one of these modifications. In particular, for any effective divisor $E$, all necessary flips in 
the $E$-MMP exist and terminate. We define an admissible flag $X= Y_0\supseteq Y_1\supseteq\dots\supseteq Y_n$ to be \emph{good}, if
    \begin{enumerate}
        \item 
            $Y_i$ is a Mori dream space for each $0\le i\le n$, and for $i=1,\ldots, n$, $Y_i$ is 
 	cut out by a global section $s_i\in H^0(Y_{i-1},\O_{Y_{i-1}}(Y_i))$, and
        \item
            for any small $\Q$-factorial modification 
            $f:Y_i\dasharrow Y_i'$, the exceptional locus $\exc(f)$
            intersects $Y_{i+1}$ properly.
    \end{enumerate}

Our main result is the following.
\begin{introthm}
Let $X$ be a Mori dream space and assume that there exists a good flag
 $Y_\bullet:  X=Y_0 \supseteq Y_1 \supseteq \cdots \supseteq Y_n$.
	Then the global Okounkov body $\Delta_{Y_\bullet}(X)$ of $X$ with respect to the 
	flag $Y_\bullet$ is a rational polyhedral cone.
\end{introthm}

In order to apply the above result to Bott-Samelson varieties, we prove the following.
\begin{introthm}
Let $Z=Z_w$ be a Bott-Samelson variety defined by a reduced expression $w$ of an element
$\overline{w}$ in the 
Weyl group of $G$.\\ 

\noindent (i) The variety $Z$ is log-Fano and hence a Mori dream space. (Theorem \ref{T: BSMDS})\\

\noindent (ii) There exists naturally a good flag $Y_\bullet$ on $Z$, 
the so called \emph{vertical flag}. (Proposition \ref{P:BSrestr})
\end{introthm}

As a consequence, all Okounkov bodies $\Delta_{Y_\bullet}(L)$ of 
line bundles $L$ over $Z$ are rational polyhedral. Using the 
desingularization $Z_w \to X_w$ of the Schubert variety $X_w$, we 
also see that line bundles 
over Schubert varieties admit rational polyhedral Okounkov bodies. 

On the representation-theoretic side, we obtain 
Okounkov bodies describing weight multiplicities. Indeed, 
the flag $Y_\bullet$ of subvarieties is $B$-invariant, which allows 
for the construction of affine subspaces $H_\mu$ mentioned before. 
We then get the following result on asymptotics of 
weight multiplicities in a section ring $R(Z, L)$. 
\begin{introthm}
Let $L$ be an effective line bundle over the Bott-Samelson 
variety $Z$. Let $H \subseteq B$ be a torus contained in 
a maximal torus of $G$ lying in $B$, and let $\mu$ 
be an rational $H$-weight. Then there exists an 
affine subspace $H_\mu$ (in $\R^{n+1}$, where $n=\mbox{dim}\, Z$) such that 
the asymptotics of the multiplicity function $m_{k\mu, k}$ defined 
above is given by
\begin{align*}
\lim_{k \rightarrow \infty} \frac{m_{k\mu, k}}{k^m}=\mbox{vol}_m(\Delta_{Y_\bullet}(L) \cap H_\mu), 
\end{align*}
where $\Delta_{Y_\bullet}(L)$ is the rational polyhedral Okounkov body of 
$L$.
\end{introthm}
If we apply this to the situation when the torus $H$ is a maximal torus, $Z$ is of maximal dimension, and 
thus admits a birational morphism $f: Z \to G/B$ to the flag variety of $G$, and $L=f^*(L_\lambda)$ is the 
pull-back of the line bundle $L_\lambda$ over $G/B$, we obtain the 
following corollary, which can be seen as an analog of Littelmann's result which
describes weight multiplicities using string polytopes.

\begin{introcor}
Let $V_\lambda \cong H^0(G/B, L_\lambda)$ be the irreducible $G$-representation of highest weight 
$\lambda$. If $\mu$ is rational weight, let $m_{k\mu, k}$, for $k\mu$ integral, denote the multiplicity 
of the weight $k\mu$ in the $G$-module $V_{k\lambda}$. Then there exists an $m \in \N$ 
such that 
\begin{align*}
 \lim_{k \rightarrow \infty} \frac{m_{k\mu, k}}{k^m}=
 \mbox{vol}_m(\Delta_{Y_\bullet}(f^*(L_\lambda)) \cap H_\mu).
\end{align*}
\end{introcor}

The present paper is organized as follows: we begin by recalling basic facts
about Okounkov bodies and Bott-Samelson varieties
in sections \ref{s: OKB} and \ref{s: BS}, respectively. 
The main result is proved in section \ref{s: main}. Finally, in section
\ref{s: appl} the result is applied to Bott-Samelson varieties, furthermore
we address there the representation-theoretic consequences. 

We work throughout over the complex numbers $\C$ as our base field.

\bigskip
{\small\noindent {\bf Acknowledgements.} } The authors would like to thank Thomas Bauer for helpful discussions and suggestions, 
as well as Jesper Funch Thomsen for having answered several questions about Bott-Samelson varieties.
We would also 
like to thank Dave Anderson for interesting remarks on an earlier version of this paper.

\section{Okounkov bodies}\label{s: OKB}

For the  convenience of the reader not familiar with the construction
of Okounkov bodies we give here a quick overview. For a thorough discussion, we
refer the reader to \cite{KK09} and \cite{LM09}. 

To a graded linear  series $W_\bullet$ on a  normal projective variety $X$ 
of dimension $n$ we want to assign a convex subset of $\R^n$ carrying 
information on $W_\bullet$. In practice, more often than not, $W_\bullet$ will
be the complete graded linear series $\bigoplus_k H^0(X,\mathcal{O}_X(kD))$ corresponding 
to an effective divisor $D$. However, at times it is convenient
to work with an arbitrary linear series associated to some divisor $D$, i.e., 
a series $W_\bullet = \{W_k\}$ of subspaces $W_k\subseteq H^0(X,\mathcal{O}_X(kD))$ satisfying
the condition $W_k\cdot W_l \subseteq W_{k+l}$. 
The construction will depend on the choice of
a valuation-like function
$$
	\nu: \bigsqcup_{k \geq 0} W_k \setminus \{0\} \to \Z^n.
$$
Instead of recounting the conditions on $\nu$, we describe a certain type
of valuation which automatically satisfies these conditions.
To this end, let 
$$
	Y_\bullet: X = Y_0 \supseteq Y_1\supseteq \dots \supseteq Y_n
$$
be a flag of irreducible subvarieties such that $\mbox{codim}_X(Y_i)=i$ 
and such that $Y_n$ is a smooth point of each $Y_i$. 
Then, for a section $s\in W_k \subseteq H^0(X,\mathcal{O}_X(kD))$, we set 
$\nu_1(s):= \mbox{ord}_{Y_1}(s)$. If we choose a local equation for $Y_1$, we 
obtain a unique section $\widetilde{s_1}\in H^0(X,\mathcal{O}_X(kD-\nu_1(s)Y_1))$ 
not vanishing identically along $Y_1$, and thus determining a section
$s_1\in H^0(Y_1,\mathcal{O}_{Y_1}(D-\nu_1(s)Y_1))$. We then set $\nu_2(s)=\mbox{ord}_{Y_2}(s_1)$, 
and proceed as before to obtain the valuation vector 
$\nu_{Y_\bullet}(s)=(\nu_1(s),\dots,\nu_n(s))$.

One then defines the valuation semi-group of $W_\bullet$ with respect to 
$Y_\bullet$ as
$$
	\Gamma_{Y_\bullet}(W_\bullet) := \left\{(\nu_{Y_\bullet}(s),m) \in \Z^{n+1} \mid 0\neq s\in W_m \right\}.
$$
Furthermore, we define the Okounkov body of $W_\bullet$ as
$$
	\Delta_{Y_\bullet}(W_\bullet) := \Sigma(\Gamma_{Y_\bullet}(W_\bullet)) \cap (\R^n\times \{1\})
$$
where $\Sigma(\Gamma_{Y_\bullet}(W_\bullet))$ denotes the closed convex cone 
in $\R^{n+1}$ spanned 
by $\Gamma_{Y_\bullet}(W_\bullet)$.

If $W_\bullet$ is the 
complete graded linear series of a divisor $D$, we write $\Delta_{Y_\bullet}(D)$ 
for the Okounkov body of $W_\bullet$. In this case, 
by \cite[Theorem 2.3]{LM09},  we have the important identity
$$
	\mbox{vol}_{\R^n}(\Delta_{Y_\bullet}(D)) = \frac 1{n!} \mbox{vol}_X(D),
$$
showing in particular, that the volume of the body $\Delta_{Y_\bullet}(D)$
is independent of the choice of the flag $Y_\bullet$.
Another important observation made in \cite{LM09} is that
even the shape of the Okounkov body $\Delta_{Y_\bullet}(D)$ for a 
divisor $D$ only depends on the numerical equivalence class of $D$.
It is therefore a natural question how the bodies $\Delta_{Y_\bullet}(D)$ change 
as $[D]$ varies in the N\'eron-Severi vector space $N^1(X)_\R$. An answer to 
this question is given in \cite[Theorem 4.5]{LM09} by proving the existence of the 
global Okounkov body: there exists a closed convex cone 
$$
	\Delta_{Y_\bullet}(X)\subset  \R^n\times N^1(X)_\R
$$ 
such that for each big divisor $D$ the fiber of the second projection over $[D]$ 
is exactly $\Delta_{Y_\bullet}(D)$. 

The concrete determination or even the description of geometric properties
of Okounkov bodies associated to some graded linear series is extremely difficult
in general. As is to be expected this will be even more true of the global 
Okounkov body of a given variety. 
In particular, it is an intriguing question under which conditions on $X$ it is possible to pick a 
flag such that the corresponding global Okounkov body is rational polyhedral. Already in \cite{LM09} this
was shown to be possible for toric varieties. Based on this evidence it is conjectured to work also 
for any Mori dream space $X$. In \cite{dh}, we introduce a possible technique to prove rational 
polyhedrality of global Okounkov bodies by constructing so-called Minkowski bases on $X$. 
Work in progress hints at the feasibility of this approach for any Mori dream space. 
In this paper however, we use a more direct strategy to prove the
rational polyhedrality of global Okounkov bodies with respect to a natural 
choice of flag for a special 
class of Mori dream spaces, namely Bott-Samelson varieties, 
which we introduce in the following section. 

Let us make a small remark on how the construction by Littelmann mentioned in
the introduction compares to Okounkov bodies. 
Littelmann's string polytopes are constructed by purely algebraic and 
combinatorial means, notably using quantum enveloping algebras of Lie algebras, 
and the result thus only shows formal analogies with the outcome of Okounkov's approach. 
However, since---by the Borel-Weil theorem---every irreducible $G$-module 
$V_\lambda$ can be realized as the space of sections $H^0(X, L_\lambda)$ of 
a line bundle $L_\lambda$ over a flag variety $X:=G/B$, where $B$ is a Borel subgroup 
of $G$, Okounkov's approach makes sense for the study of 
asymptotics of weight spaces in the section ring $R(X, L_\lambda)$. 
For the approach to work, the flag $Y_\bullet$ should 
consist of $H$-invariant subvarieties. A natural candidate 
for such a flag would then be a flag of Schubert varieties, and indeed 
this approach was taken by Kaveh in \cite{K11}. For technical 
reasons, notably for having a flag of Cartier divisors, Kaveh 
passed to a Bott-Samelson resolution $Z \to X$ of $X$, pulled 
back $L_\lambda$ to $Z$, and replaced the flag $Y_\bullet$ by 
a flag $Z_\bullet$ of (translations of) Bott-Samelson subvarieties of $Z$. 
The main result in \cite{K11} is that Littelmann's 
string bases can be interpreted in terms of a $\Z^n$-valued 
valuation on the function field $\C(Z)$ of $Z$, depending on 
the flag $Z_\bullet$. This valuation, however, differs from 
those introduced by Okounkov: whereas orders of vanishing 
of a regular function $f$ are described in local coordinates $x_1,\ldots, x_n$  
by the smallest monomial term of $f(x)=\sum_{a \in \N^n} c_a x^a$, 
with respect to some ordering of the variables 
$x_1,\ldots, x_n$, Kaveh's valuation is locally defined by the 
highest monomial term. In geometric language, this 
valuation thus tells how often $f$ can be differentiated in the 
various directions defined by the $x_i$ in the given order.
It still remains an open problem to interpret Littelmann's string polytopes 
as Okounkov bodies, or indeed, more generally, to construct rational polyhedral 
Okounkov bodies for line bundles over flag varieties using some $H$-invariant 
flag $Y_\bullet$. 

\section{Bott-Samelson varieties}\label{s: BS}

Let us recall the basics of Bott-Samelson varieties, following \cite{LT04}.

Let $G$ be a connected complex reductive group, let $B \subseteq G$ be a Borel subgroup, and let 
$W$ be the Weyl group of $G$. If $s_i \in W$ is a simple reflection, let $P_i$ denote 
the associated minimal parabolic subgroup containing $B$. Then the quotient space 
$P_i/B$ is isomorphic to $\mathbb{P}^1$. 
For a sequence $w=(s_1,\ldots, s_n)$ (where the $s_i$ are not necessarily distinct), let 
$P_w:=P_1 \times \cdots \times P_n$ be the product of the corresponding parabolic subgroups, 
and consider the right action of $B^n$ on $P_w$ given by
\begin{align*}
(p_1,\ldots, p_n)(b_1,\ldots, b_n):=(p_1b_1, b_1^{-1}p_2b_2,b_2^{-1}p_3b_3,\ldots, b_{n-1}^{-1}p_nb_n).
\end{align*}
The Bott-Samelson variety $Z_w$ is the quotient
\begin{align*}
Z_w:=P_w/B^n.
\end{align*}
An alternative description of this quotient can be given as follows. Suppose that $X$ and $Y$ are 
two varieties, such that $X$ is equipped with a right action and $Y$ with a left action of 
$B$. Consider the right action of $B$ on the 
product given by
\begin{align*}
(x,y).b:=(xb,b^{-1}y), \quad (x, y) \in X \times Y, \quad b \in B,
\end{align*}
and let $X \times_B Y:=(X \times Y)/B$ denote the quotient space. 
Then the map $X \times_B Y \to X/B, \quad [(x,y)] \mapsto xB$ exhibits $X \times_B Y$ as a fiber bundle 
over $X/B$ and with fiber $Y$. Now, we can alternatively describe $Z_w$ as 
\begin{align*}
Z_w=(P_1 \times_B \cdots \times_B P_n)/B,
\end{align*}
where $B$ acts on the right on $P_1 \times_B \cdots \times_B P_n$ by
\begin{align*}
[(p_1,\ldots, p_n)].b:=[(p_1,\ldots, p_{n-1}, p_nb)], \quad (p_1,\ldots, p_n) \in P_w, \quad b \in B.
\end{align*}
As a consequence, using the fact that each quotient $P_i/B$ is isomorphic to $\mathbb{P}^1$, 
$Z_w$ is given as an iteration of $\mathbb{P}^1$-bundles. To describe this structure in more 
detail, let, for $j \in \{1,\ldots, n \}$,  $w[j]$ denote the truncated 
sequence $(s_1,\ldots, s_{n-j})$, and let $Z_{w[j]}:=P_{w[j]}/B^{n-j}$ denote 
the associated Bott-Samelson variety. 
Then the projections $P_w \to P_{w[j]}$ are $B^n$-equivariant, where 
$B^n$ acts on $P_{w[j]}$ by the factor $B^{n-j}$, and thus induce a projections 
$\pi_{w[j]}: Z_w \to Z_{w[j]}$, 
which can be factorized as a sequence of $\mathbb{P}^1$-fibrations 
\begin{align*}
\pi_{w[1]}: Z_w \to Z_{w[1]} \to \cdots \to Z_{w[j]}.
\end{align*}

Let $\pi: Z_w \to P_1/B$ denote the composition of all these projections, i.e., 
$\pi$ is the projection morphism onto $P_1/B$ defined by the description of 
$Z_w$ as the bundle
\begin{align*}
Z_w:=P_1 \times_B (P_2 \times_B \times \cdots \times_B P_n).
\end{align*}

Now, each $\mathbb{P}^1$-bundle admits a natural section as follows. Let 
$w(j):=(s_1,\ldots, \hat{s_j}, \ldots, s_n)$, 
so that $P_{w(j)}$ embeds naturally as a subgroup of $P_w$. The embedding 
$\sigma_{w, j}^0: P_{w(j)} \to P_w$ is $B^{n-1}$-equivariant, and thus 
induces an embedding
\begin{align*}
	\sigma_{w, j}: Z_{w(j)} \to Z_{w}
\end{align*}
of $Z_{w(j)}$ as a divisor in $Z_w$ such that the divisors $Z_{w(j)}, \, j=1,\ldots, n$, intersect 
transversely in a point.
In particular, $\sigma_{w,n}: Z_{w(n)} \cong Z_{w[1]} \to Z_w$ defines a section of 
the $\mathbb{P}^1$-bundle $\pi_{w[1]} Z_w \to Z_{w[1]}$, and identifies $Z_{w[1]}$ with a divisor 
which is transversal to the fibers of $\pi_{w[1]}$.
Now, the Picard group $\mbox{Pic}(Z_w)$ splits as the direct sum 
\begin{align*}
	\mbox{Pic}(Z_w) \cong \mbox{Pic}(Z_{w[1]}) \oplus \Z, 
\end{align*}
where $\Z$ is identified with the subgroup generated by the line bundle 
$\mathcal{O}_{Z_w}(Z_{w(n)})$. Iterating the above splitting yields that 
the line bundles $\mathcal{O}_{Z_w}(Z_{w(j)}), \, j=1,\ldots, n$, 
define a basis for $\mbox{Pic}(Z_w)$. Clearly, they are all effective. 
Conversely, if $w$ is a reduced sequence, i.e., if the length of the product 
$\overline{w}:=s_1\cdots s_n$ equals 
$n$, a divisor 
\begin{align*}
m_1Z_{w(1)}+\cdots +m_n Z_{w(n)}, \quad m_1,\ldots, m_n \in \Z,
\end{align*}
is effective if and only if $m_1,\ldots, m_n \geq 0$ (cf. \cite[Prop. 3.5]{LT04}). 
The basis $\{Z_{w(1)},\ldots, Z_{w(n)}\}$ 
is called the \emph{effective basis} for $\mbox{Pic}(Z_w)$.
Notice that, since $Z_{w(n)}$ defines a section of the bundle $\pi_{w[1]}$, 
the restricted divisors
\begin{align}
Z_{w(1)} \cdot Z_{w(n)}, \ldots, Z_{w(n-1)} \cdot Z_{w(n)}
\label{E: rbasisbs}
\end{align}
form the effective basis for $Z_{w[1]} \cong Z_{w(n)}$.

\subsection{The vertical flag}\label{ss:VerticalFlag}
We also recall the so-called $\mathcal{O}(1)$-basis for $\mbox{Pic}(Z_w)$ 
defined as follows. Each product $P_1 \times \cdots \times P_k$, $k=1,\ldots, n$, 
defines a morphism 
\begin{align*}
\varphi_k: Z_{w[n-k]} \to G/B, \quad [(p_1,\ldots, p_k)] \mapsto p_1p_2\cdots p_k B.
\end{align*}
Put $\mathcal{O}_{w[n-k]}(1):=\varphi_k^*L_{\omega_\alpha}$, 
where $\alpha$ is simple root corresponding to the simple reflection $s_k$,  
$\omega_\alpha$ is its fundamental weight, viewed as a character of $B$, and 
$L_{\omega_\alpha}:=G \times_B \C$ is the associated line bundle over the flag variety $G/B$. 
Let $\mathcal{O}_k(1):=\pi_{w[n-k]}^*\mathcal{O}_{w[n-k]}(1)$. The line 
bundles $\mathcal{O}_1(1),\ldots, \mathcal{O}_n(1)$, being pullbacks of 
globally generated line bundles, are then globally generated and form 
a basis for $\mbox{Pic}(Z_w)$. Moreover, a line 
bundle $\mathcal{O}_1(1)^{m_1} \otimes \cdots \otimes \mathcal{O}_n(1)^{m_n}$ is 
very ample (nef) if and only if $m_1,\ldots, m_n>0$ ($m_1,\ldots, m_n \geq 0$) (
\cite[Thm. 3.1]{LT04}).
Notice here that the morphism $\varphi_k$ above is 
induced by the $B$-equivariant multiplication map 
\begin{align*}
P_1 \times \cdots \times P_k \to G, \quad (p_1,\ldots, p_k) \mapsto p_1p_2\cdots p_k, 
\end{align*}
where $B$ acts on $P_1 \times \cdots \times P_k$ by the right multiplication on the 
factor $P_k$. We can therefore also view the line 
bundle $\mathcal{O}_{w[n-k]}(1)$ over $Z_{w[n-k]}$ as the 
product 
\begin{align*}
 \mathcal{O}_{w[n-k]}(1)= P_1 \times \cdots \times P_k \times_{B^k} \C, 
\end{align*}
where $B^k$ acts on the product $P_1 \times \cdots \times P_k$ by 
\begin{align*}
(p_1,\ldots, p_k).(b_1,\ldots, b_k):=(p_1b_1, b_1^{-1}p_2b_2,\ldots, b_{k-1}^{-1}p_kb_k), 
\end{align*}
and on $\C$ by the character
\begin{align*}
(b_1,\ldots, b_k) \mapsto \omega_\alpha(b_k).
\end{align*}
Thus, the sections of the sheaf $\mathcal{O}_{w[n-k]}(1)$ over an open set $U \subseteq Z_{w[n-k]}$ 
correspond to the the regular functions $f$ on the $B^k$-invariant open subset 
$$\widetilde{U}_k:=\{(p_1,\ldots, p_k) \in P_1 \times \cdots \times P_k \mid 
[(p_1,\ldots, p_k)] \in U\}$$
satisying the $B^k$-equivariance property 
\begin{align*}
 &f(p_1b_1, b_1^{-1}p_2b_2,\ldots, b_{k-1}^{-1}p_kb_k)=\omega_\alpha(b_k)^{-1}f(p_1,\ldots, p_k),\\
 &(p_1,\ldots, p_k) \in \widetilde{U}_k, \quad (b_1,\ldots, b_k) \in B^k.
\end{align*}
It follows that sections of the sheaf $\mathcal{O}_k(1)$ over the open subset $\pi_{w[n-k]}^{-1}(U)$ then
correspond to the regular functions $f$ on the $B^n$-invariant open subset 
$$\widetilde{U}:=\{(p_1,\ldots, p_n) \in P_1 \times \cdots \times P_n \mid [(p_1,\ldots, p_k)] \in U\}$$
satisfying the $B^n$-equivariance property
\begin{align*}
 &f(p_1b_1, b_1^{-1}p_2b_2,\ldots, b_{n-1}^{-1}p_nb_n)=\omega_\alpha(b_k)^{-1}f(p_1,\ldots, p_n),\\ 
 &(p_1,\ldots, p_n) \in \widetilde{U}, \quad (b_1,\ldots, b_n) \in B^n.
\end{align*}
In other words, $\mathcal{O}_k(1)$ is the line bundle 
\begin{align*}
\mathcal{O}_k(1)=P_1 \times \cdots \times P_n \times_{B^n} \C,
\end{align*}
where $B^n$ acts on $\C$ by the 
character
\begin{align*}
\xi_k: B^n \rightarrow \C^\times, \quad \xi_k(b_1,\ldots, b_n):=\omega_\alpha(b_k).
\end{align*}
Since the $\mathcal{O}_k(1)$, $k=1,\ldots, n$, form a basis for the Picard group, we 
see that each line bundle corresponds to a unique character $\xi_1^{m_1}\cdots \xi_n^{m_n}$, 
for an $(m_1,\ldots, m_n) \in \Z^n$ in such a way that the 
sections of the sheaf $\mathcal{O}_1(1)^{m_1}\otimes \cdots \otimes \mathcal{O}_n(1)^{m_n}$
over an open subset $U \subseteq Z_w$ correspond to the 
regular functions $f$ on the $B^n$-invariant open subset
$$\widetilde{U}:=\{(p_1,\ldots, p_n) \in P_1 \times \cdots \times P_n \mid [(p_1,\ldots, p_k)] \in U\}$$
satisfying the $B^n$-equivariance property
\begin{align}
 &f(p_1b_1, b_1^{-1}p_2b_2,\ldots, b_{n-1}^{-1}p_nb_n)=\xi_1(b_1)^{-m_1} \cdots \xi_n(b_n)^{-m_n}f(p_1,\ldots, p_n), 
 \nonumber\\
 &(p_1,\ldots, p_n) \in \widetilde{U}, \quad (b_1,\ldots, b_n) \in B^n. \label{E: equivprop}
\end{align}
Thus, every line bundle $L$ on $Z_w$ is of the form 
\begin{align*}
L=P_1 \times \cdots \times P_n \times_{B^n} \C,
\end{align*}
where $B^n$ acts on $\C$ by the character $\xi_1^{m_1} \cdots \xi_n^{m_n}$, for a unique 
$(m_1,\ldots, m_n) \in \Z^n$. Let $\mathcal{O}(m_1,\ldots, m_n)$ denote the line bundle 
corresponding to $(m_1,\ldots, m_n)$. In particular, each line bundle $L=\mathcal{O}(m_1,\ldots, m_n)$ 
admits an action of $P_1$ as bundle automorphisms by 
\begin{align*}
p.[((p_1,\ldots, p_n), z)] \mapsto [((pp_1,p_2,\ldots, p_n), z)], 
\end{align*}
which is clearly well-defined since the left multiplication of $P_1$ on the 
$P_1$-factor in $P_1 \times \cdots \times P_n$ commutes with the $B^n$-action on this product. 
The induced representation of $P_1$ on the space of global sections $H^0(Z_w, \mathcal{O}(m_1,\ldots, m_n))$ is given by 
\begin{align}
(p.f)(p_1,\ldots, p_n):=p.(f(p^{-1}p_1,p_2,\ldots, p_n)), 
\label{E: repsections}
\end{align}
for $\quad p \in P_1$ and $(p_1,\ldots, p_n) \in P_1 \times \cdots \times P_n$, 
where we have used the identification \eqref{E: equivprop} of sections with equivariant regular 
functions on $P_1 \times \cdots \times P_n$. Clearly, the $B^n$-equivariance property 
is preserved by the left action of $P_1$, so that this indeed defines a representation of 
$P_1$ on $H^0(Z_w, \mathcal{O}(m_1,\ldots, m_n))$.

Consider now the fibration 
$$
	\pi: Z_\omega \to P_1/B\cong \mP^1 
$$
given by mapping $[(p_1,\dots,p_n)]$ to the class $[p_1]=p_1\cdot B$. 
All of its fibers are isomorphic to the Bott-Samelson variety 
$Z_{[s_2,\dots,s_n]}$. Note that $B$ operates on the quotient $P_1/B$ as
the upper triangular matrices acts on $\mP^1$, 
$$
	\left(\begin{matrix}a&b\\0&c\end{matrix}\right)  [z_0:z_1] = [az_0+bz_1: cz_1],
$$
and thus with exactly one fixed point $p_0$. 
Denote by $Y$ the fiber over $p_0$. 
Note also that for any $p\in P_1$ and $[p_1,\dots,p_n]\in Z_\omega$ we have
\be
	\pi(p.[p_1,\dots,p_n]) &=& \pi([pp_1,p_2,\dots,p_n])\\
												&=& pp_1B = p.\pi([p_1,\dots,p_n]), 
\ee
i.e, $\pi$ is $P_1$-equivariant, and hence $P_1$  
acts as automorphisms of the fiber bundle $\pi$.

We can reiterate this construction to obtain a natural 
flag of Bott-Samelson varieties
$$
	Z_\omega\supseteq Y_1\supseteq \dots\supseteq Y_n
$$
where $Y_i$ is given as the fiber of  the corresponding $\mathbb{P}^1$-bundle 
$Y_{i-1} \to \mathbb{P}^1$ over the $B$-fixed point.
We call this flag the \emph{vertical flag} on $Z_\omega$.

\subsection{Bott-Samelson varieties as Mori dream spaces}
Now assume that $w$ is a reduced sequence. Recall that 
the product map $P_w \to G, (p_1,\ldots, p_n) \to p_1 \cdots p_n$ induces 
a morphism 
\begin{align*}
p_w: Z_w \to Y_{\overline{w}}:=\overline{B\overline{w}B}
\end{align*}
into the Schubert subvariety $Y_{\overline{w}}$ of the flag variety $G/B$, and that 
this morphism is in fact birational. Moreover, it is $B$-equivariant 
with respect to the left action of $B$ on $Z_w$ defined 
by
\begin{align*}
b[(p_1,\ldots, p_n)]:=[(bp_1,p_2, \ldots, p_n)], \quad (p_1,\ldots, p_n) \in P_w, \quad b \in B.
\end{align*}
In particular, if $\overline{w}$ is the longest element of the Weyl 
group, $p_w$ defines a birational map $Z_w \to G/B$. 

The following theorem is the crucial result for our application of the theory 
from the following section.
\begin{thm} \label{T: BSMDS}
Let $G$ be a complex reductive group with Weyl group $W$, and let $Z=Z_w$ be a Bott-Samelson 
variety defined by a reduced sequence $w$ of simple reflections.
Then $Z$ admits a divisor $\Delta$ such that $(Z, \Delta)$ is a log Fano pair. In particular, 
$Z$ is a Mori dream space.
\end{thm}

\begin{proof}
Let $Y=G/B$, and let Let $D_\rho$ be the divisor on $Y_{\overline{w}}$ which corresponds to 
the restriction to $Y_{\overline{w}}$ of the square root of the anticanonical bundle 
of $Y$. Then $D_\rho$ is an ample divisor on $Y_{\overline{w}}$, so that 
$p_w^*(D_\rho)$ is a nef divisor on $Z$.

In order to facilitate the notation, let 
$\{D_1,...,D_n\}$ be the basis of effective divisors for $\mbox{Pic}(Z)$.
Now choose integers $a_1,...,a_n>0$ so that $\sum_{i=1}^n a_iD_i$
is an ample divisor. Then, for every $N>0$, the divisor 
$p_w^*(-D_{\rho})-\sum_{i=1}^n a_i/N D_i$
is anti-ample. 
Now let $N \in \N$ be so big that $a_i/N<1$ for every $i$, and
put $$\Delta:=\sum_{i=1}^n(1-a_i/N) D_i.$$
If $K_Z$ is the canonical divisor of $Z$, we then have that 
\begin{align*}
K_Z+\Delta=\pi^*(-D_{\rho})-\sum_{i=1}^n D_i+\sum_{i=1}^n(1-a_i/N) D_i
=\pi^*L_{-\rho}-\sum_{i=1}^n (a_i/N) D_i
\end{align*}
(cf. \cite[Lemma 5.1]{LT04}) 
is anti-ample. Since $Z$ is nonsingular, and all subsets of the set of smooth divisors 
$\{D_1,\ldots, D_n\}$ intersect transversely and smoothly, the pair $(Z, \Delta)$ thus 
defines a log Fano pair.
\end{proof}

\begin{rem}
In the context of the above theorem it is worth mentioning
an analogous result by Anderson and Stapledon (\cite{AS}) on the log Fano property 
of Schubert varieties. 
\end{rem}

%

\section{Good flags on Mori dream spaces}\label{s: main}

In this section we prove the main theorem of this paper. The main
objective is to establish conditions on a flag on a Mori dream space, such
that its global Okounkov body is rational polyhedral.

First let us recall that a Mori dream space $X$ is a normal $\Q$-factorial variety 
such that  $\mbox{Pic}(X)_\Q\cong N^1(X)_\Q$ and with a Cox ring $\Cox(X)$ which is a finitely 
generated $\C$-algebra. We make use of the theory of Mori dream spaces 
developed by Hu and Keel in \cite{hk} and we refer the reader to this beautiful paper 
for a detailed investigation of Mori dream spaces. 

Note that  for any effective divisor $D$ on
a Mori dream space $X$, the ring of sections $R(X,D):=\bigoplus_{k \geq 0} H^0(X,\mathcal{O}_X(kD))$ is
finitely generated, so we obtain a natural rational map 
$$
	f_D: X \dasharrow \mbox{Proj}(R(X,D))
$$
which is regular outside the stable base locus of $D$. One obtains an
equivalence relation of effective divisors as follows: two effective divisors $D$ and $D'$ 
are \emph{Mori-equivalent} if  up to isomorphism they yield the same rational maps.
Hu and Keel prove (\cite[Proposition 1.11]{hk}) that there are only finitely many 
equivalence classes, indexed by 
contracting rational maps $f:X\dasharrow X'$ and that the closure $\Sigma_f$ of a 
maximal dimensional equivalence class can be described as the closed convex 
cone spanned by $f$-exceptional rays together with the face $f^\ast(\Nef(X'))$ 
of the moving cone. These subcones $\Sigma_f$, which 
decompose the pseudo-effective cone $\Eff(X)$, 
are in the remainder of this paper following \cite{hk} referred to  as \emph{Mori-chambers}.

Now, let $f: X \dasharrow X'$
be a small $\Q$-factorial modification defining an isomorphism 
$$f\mid_U: U \rightarrow V$$
between open subsets $U \subseteq X, V \subseteq X'$ with 
complements of codimension at least two. 
Then $f^*$ induces an isomorphism of 
pseudo-effective cones $\overline{\textrm{Eff}}(X) \cong \overline{\textrm{Eff}}(X')$ as well as an isomorphism 
$\textrm{Cox}(X) \cong \textrm{Cox}(X')$ of 
Cox rings. We further recall that $f$ is induced by GIT in the 
following manner. 

Let $R=\mbox{Cox}(X)$. The variety $X$ can be written as the GIT-quotient 
$X=\spec(R)^{ss}(\chi)//G$, where $G$ is the complex torus of rank equal to 
the Picard number of $X$, and $\spec(R)^{ss}(\chi)$ denotes the 
set of semistable points in $\spec(R)$ with respect to the character $\chi$ 
of $G$. We even have that $\spec(R)^{ss}(\chi)=\spec(R)^s(\chi)$; the 
set of $\chi$-semistable points equals the set of $\chi$-stable points, so that the 
quotient is even a geometric quotient.  Also $X'$ is a geometric quotient 
of the set of stable points with respect to a character $\chi'$: 
$X'=\spec(R)^s(\chi')//G$. Let $\pi_\chi: \spec(R)^s(\chi) \to X$ and
$\pi_{\chi'}: \spec(R)^s(\chi') \to X'$ denote the respective 
quotient morphisms. In both cases the sets of 
unstable points (with respect to $\chi$ and $\chi'$) are 
of codimension at least two in $\spec(R)$, and the 
rational map $f$ is induced on the level of quotients by the 
inclusion 
$$\spec(R)^s(\chi) \cap \spec(R)^s(\chi') \subseteq \spec(R)^s(\chi')$$
of the subset of common stable points into the set of $\chi'$-stable points.
In particular, the exceptional locus of $f$ equals the complement of the 
domain of definition of $f$ and is given as the image of 
the $\chi'$-unstable and $\chi$-stable points; 
$$\mbox{exc}(f)=\pi_\chi(\spec(R)^s(\chi) \cap \spec(R)^{us}(\chi')),$$
so that $f$ induces an isomorphism 
\begin{align*}
f\mid_U: U \stackrel{\cong}{\to} V, \quad
U&:=\pi_\chi(\spec(R)^s(\chi) \cap \spec(R)^s(\chi')),\\ 
V&:=\pi_{\chi'}(\spec(R)^s(\chi) \cap \spec(R)^s(\chi')).  
\end{align*}

Now let $f: X \dasharrow X'$ be a small $\Q$-factorial modification as above, and 
assume that $Y \subseteq X$ is an irreducible hypersurface 
given as the zero set $Y=Z(s)$ of a section $s \in H^0(X, L)$ of some 
line bundle $L$. Let $L':=(f^{-1})^*L$ denote the corresponding line 
bundle on $X'$, let $s' \in H^0(X', L')$ be the section of $L'$ corresponding to $s$, 
and put $Y':=Z(s')$. The restriction of $f$ to $Y$ then 
defines a birational map $$f_Y: Y \dasharrow Y', $$
yielding the commuting diagram
\begin{align*}
\xymatrix{
 X  \ar@{-->}[r]^{f} & X'\\
 Y \ar[u] \ar@{-->}[r]^{f_Y} & Y', \ar[u]
 }
\end{align*}
where  the vertical arrows denote the respective inclusion morphisms. 
We now have the following simple but useful lemma.
\begin{lemma}
a) The hypersurface $Y' \subseteq X'$ is irreducible.\\

b) If $Y$ intersects the exceptional locus $\mbox{exc}(f)$ properly, 
then $f_Y$ induces an isomorphism of open subsets
\begin{align*}
U \cap Y \stackrel{\cong}{\to} V \cap Y', 
\end{align*}
where $\mbox{codim}_Y(Y \setminus (U \cap Y)) \geq 2$.
\end{lemma}

\begin{proof}
Consider the restriction $(\pi_\chi)_Z: Z \to Y$ of $\pi_\chi$ to the 
preimage $Z:=\pi_\chi^{-1}(Y) \subseteq \spec(R)^s(\chi)$ of $Y$. 
The fibers of $(\pi_\chi)_Z$, being closed $G$-orbits, are irreducible, and 
they are all of the same dimension, $\mbox{dim}\, G$, since 
all stabilizers of points in $\spec(R)^s(\chi)$ are finite. Since $Y$ is 
irreducible, it then follows that $Z$ is irreducible, and hence also 
the closure $\overline{Z}$ of $Z$ in $\spec(R)$. Moreover, 
$\overline{Z}$ can be written as the zero-set $\overline{Z}=V(f)$, where 
we identify the section $s \in H^0(X, L) \subseteq R$ with a function $f$ on $\spec(R)$. Thus, 
the zero-set $V(f) \subseteq \spec(R)$ is irreducible. 
Now, the zero set $Y'=Z(s')$ is given as 
$Z(s')=\pi_{\chi'}(V(f) \cap \spec(R)^s(\chi'))$ and is  
clearly also irreducible. This proves a).

As for b), the argument above shows that $f_Y$ in general defines an 
isomorphism 
\begin{align*}
U \cap Y=U \cap Z(s) \stackrel{\cong}{\to} V \cap Z(s')=V \cap Y.' 
\end{align*} 
The condition that $Y$ intersect the exceptional locus $\mbox{exc}(f)$ properly 
then shows that the codimension of $Y \setminus (U \cap Y)$ in $Y$ is 
at least two. This finishes the proof. 
\end{proof}

We now turn to Okounkov bodies on a Mori dream space $X$ equipped with an
admissible flag $Y_\bullet$. Our strategy is to deduce properties of the global 
Okounkov body of $X$ from those of Okounkov bodies of line bundles restricted to $Y_1$ 
and to argue inductively.
We are thus particularly interested in a comparison 
of the Okounkov bodies of a graded linear series coming from restricting sections to $Y_1$ and 
of a restricted line bundle. More concretely, for a divisor $D$ on $X$ we consider the restriction map
\begin{align*}
	R: \bigoplus_{k} H^0(X, \mathcal{O}_X(kD)) \to  
			\bigoplus_{k} H^0(Y, \mathcal{O}_Y(k D \cdot Y))
\end{align*}
and hope for an identity
\begin{align}\label{E: idokbodies}
	\Delta_{Y^1_\bullet}(\mbox{im} \, R)=\Delta_{Y^1_\bullet}(D \cdot Y),
\end{align}
where $Y^1_\bullet$ denotes the flag $Y_2\supseteq\ldots\supseteq Y_n$ on $Y_1$.
Note that in case $D$ is ample the above identity holds. This follows from the exact sequence
$$
	H^0(X,\O_X(mD))\to H^0(Y,\O_Y(mD))\to H^1(X,\O_X(mD-Y))
$$  
together with the fact that for large $m$ the last cohomology group is 
trivial by Serre's vanishing theorem. In order to get the desired identity 
for any movable divisor $D$, we will consider the corresponding small 
modification. 

\begin{prop}\label{P:restriction}
Let $X$ be a Mori dream space, and let $Y_\bullet$ be an admissible flag 
of normal subvarieties, write $Y:=Y_1$, and let 
 $Y^1_\bullet$ denote the admissible flag 
$$Y_n \subseteq \cdots \subseteq Y_1=Y$$ of subvarieties of $Y$. 

Let $f: X \dasharrow X'$ be a SQM, 
and let $f_Y: Y \dasharrow Y'$ be the 
induced birational morphism. 
Assume that there exist open subsets $U\subset Y$ and $V\subset Y'$ with 
codim${}_{Y}(Y\setminus U)\ge2$ such that $f_Y:U\to V$ is an isomorphism.

If $D'$ is a nef divisor on $X'$ and $D:=f^*(D')$ is 
the corresponding divisor on $X$, then the identity
\begin{align*}
\Delta_{Y^1_\bullet}(\mbox{im} \, R)=\Delta_{Y^1_\bullet}(D \cdot Y)
\end{align*}
of Okounkov bodies holds.
\end{prop}

\begin{proof}
Since $S:=\mbox{im} R$ is a graded linear subseries of  
$\bigoplus_{k} H^0(Y, \mathcal{O}_Y(k D \cdot Y))$, the claimed identity 
of Okounkov bodies follows if we can show that both series have the same volume. 

We have a commutative diagram
\begin{align*}
\xymatrix{
 X  \ar@{-->}[r]^{f} & X'\\
 Y \ar[u] \ar@{-->}[r]^{f_Y} & Y', \ar[u]
 }
\end{align*}
where  the vertical arrows denote the respective inclusion morphisms 
By assumption,
there are open subsets $U\subset Y$ and $V\subset Y'$ with 
codim${}_{Y}(Y\setminus U)\ge2$ such that $f_Y:U\to V$ is an isomorphism.



Let us first assume that $D'$ is ample. 
Denote the restricted divisors $D\cdot Y$ and $D'\cdot Y'$ by $D_Y$ and $D'_{Y'}$, respectively.

Since $D$ defines the map $f$, we have the identity
$$
	D = f^\ast(\mathcal O_{X'}(1)).	
$$
This restricts to $U$ as the identity $f_Y^\ast(\mathcal O_{Y'}(1))|_U =D_Y|_U$. 
By assumption, $Y\setminus U$ has codimension at least $2$ in $Y$, 
so that we get 
$$
	f_Y^\ast(\mathcal O_{Y'}(1)) = D_Y.
$$
Being ample, $\mathcal O_{Y'}(1)$ does not have a divisorial base component in $Y'\setminus V$,
so we can represent it as a divisor $A = \overline{A\cap V}$. Therefore, under $f_Y^{-1}$ the divisor 
$D_Y$ pulls back to $A$.

Since $Y'$ might not be normal we consider the normalization
$$
	\pi: \widetilde{Y'} \to Y'.
$$
Note that $\pi$ 
defines an isomorphism $\pi^{-1}(V) \to V$ since $V$ is contained in the normal locus of $Y'$. 
In particular, we have the identity
$$
	(f_Y^{-1} \circ \pi)^\ast(D_Y) = \pi^\ast f_Y^{-1}{}^\ast(D_Y) = \pi^\ast(A).
$$
Since $f_Y^{-1}\circ \pi$ is a contracting birational map between normal varieties, 
this implies the identity of volumes
$$
	\vol(D_Y) = \vol(\pi^\ast(A)).
$$
Since $\pi$ is a birational morphism, the right hand side is just $\vol(A)$. 
On the other hand, 
$$
	\vol(A) = \vol(\mathcal O_{Y'}(1)) = \vol(S),
$$
since $Y'=\mbox{Proj}(S)$. This proves the claim in  case $D'$ is ample.

If $D'$ is merely nef, write $[D']$ as a limit $[D']=\lim_{i \rightarrow \infty} [D'_i]$
of numerical equivalence classes of ample divisors $D'_i, i \in \N$. Put 
$D_i:=f^*D_i$. Then $[D]=\lim_{i \rightarrow \infty} [D_i]$. Now, let 
$a \in \Delta_{Y^1_\bullet}(D \cdot Y)$. 
Then $(a, [D \cdot Y]) \in \Delta_{Y^1_\bullet}(Y)$. Now choose points $a_i \in 
\Delta_{Y^1_\bullet}(Y \cdot D_i)$ so that 
$(a, [D \cdot Y])=\lim_{i \rightarrow \infty}(a_i, [D_i \cdot Y])$. 
By the above, the identity \eqref{E: idokbodies} holds when $D$ is replaced by $D_i, i \in \N$. 
Hence, by \cite[Theorem 4.26]{LM09}, $$((0,a_i), [D_i]) \in \Delta_{Y_\bullet}(X)$$ for each $i$, so 
that $$((0,a), [D])=\lim_{i \rightarrow \infty} ((0,a_i), [D_i]) \in \Delta_{Y_\bullet}(X),$$ i.e., 
$(0,a) \in \Delta_{Y_\bullet}(D)$. This shows that the 
identity \eqref{E: idokbodies} holds for an arbitrary nef divisor $D'$ on $X'$. 

\end{proof}


In order to apply the above proposition to obtain information
on the structure of the global Okounkov body of a Mori dream space, 
we need the following construction formulated in a more general context.
Here $Y_\bullet$ can be any admissible flag on a normal projective variety $X$.

If $s \in H^0(X, \mathcal{O}_X(m_1D_1+\cdots +m_nD_n))$ is a section 
which does not vanish on $Y_1$, so that $\nu(s)=(\nu_1(s),\ldots, \nu_n(s))$ with 
$\nu_1(s)=0$, then the restriction of $s$ to $Y_1$ defines a section 
of the line bundle $\mathcal{O}_{Y_1}(D \cdot Y_1)$ over $Y_1$ with value 
$\nu^1(s)=(\nu_2(s),\ldots \nu_n(s))$ with respect to the truncated
flag 
\begin{align}
Y_n \subseteq \cdots \subseteq Y_{1}
\label{E: shortflag}
\end{align}
on $Y_1$.

For a finite set $F_1,\ldots, F_r$ of movable divisors on $X$,  let 
$$	
	\Gamma(F_1,\ldots, F_r) \subseteq \mbox{Mov}(X)
$$ 
be the semigroup generated by the divisors $F_1,\ldots, F_r$, and 
let 
$$
C(F_1,\ldots, F_r) \subseteq \mbox{Mov}(X)
$$
be the cone generated by $F_1,\ldots, F_r$.
Define the semigroups
\begin{align*}
S(F_1,\ldots, F_r):=\{(\nu(s),  [D]) \in \N_0^n \times \Gamma(F_1,\ldots, F_r) \mid s \in H^0(X, \mathcal{O}_X(D))&,&\\
D \in \Gamma(F_1,\ldots, F_r), \, \nu_1(s)=0\}
\end{align*}
and
\begin{align*}
S_1(F_1,\dots,F_r):=\{(\nu^1(s), [D \cdot Y_1]) \in \N_0^{n-1} \times N^1(Y_1)_\R \mid
[D] \in\Gamma(F_1,\ldots, F_r) &,&\\ 
s \in H^0(Y_1, \mathcal{O}_{Y_1}(D \cdot Y_1))\},
\end{align*}
as well as the morphism of semigroups
\begin{align*}
q_0: S \rightarrow S_1, \quad q_0(\nu(s), [D]):=(\nu^1(s), [D  \cdot Y_1]),
\end{align*}
which extends to the linear map
\begin{align}
q: \qquad \R^n \oplus V(F_1,\ldots, F_r) &\to \R^{n-1} \oplus N^1(Y_1)_\R, \label{E: linmapq}\\
((x_1,\dots,x_n), [D]) &\mapsto ((x_2,\dots, x_n), [D \cdot Y_1]), \nonumber
\end{align}
where $V(F_1,\ldots, F_r) \subseteq N^1(X)_\R$ is the $\R$-vector space 
generated by the numerical equivalence classes $[F_1],\ldots, [F_r]$. Furthermore, 
we denote by $C(S(F_1,\dots,F_r))$ and $C(S_1(F_1,\dots,F_r))$
the closed convex cones in $\R^n\times N^1(X)_\R$ and $\R^{n-1} \times N^1(Y_1)_\R$ 
spanned by the semigroups
$S(F_1,\dots,F_r)$ and $S_1(F_1,\dots,F_r)$, respectively.

We now recall that for a Mori dream space $X$
the pseudo-effective cone $\overline{\mbox{Eff}}(X)$ is the union of finitely 
many Mori chambers, $\Sigma_1,\ldots, \Sigma_m$, where each 
Mori chamber $\Sigma_j$ is the convex hull of finitely many integral divisors $D^j_1,\ldots, D^j_{\ell_j}$. 
More concretely, by \cite[Proposition 1.1]{hk}, the chambers are in correspondence to contracting
birational maps $f_j:X\dasharrow X_j$ with image a Mori dream space, and are given
as the convex cone spanned by $f_j^*(\Nef(X_j))$ together with the rays spanning the 
exceptional locus $\exc(f_j)$. The corresponding decomposition of a divisor $D\in\Sigma_j$ 
is exactly its decomposition into its fixed and movable parts. We can thus reorder  the 
divisors spanning each chamber in such a way that the first $n_j$ of them are movable 
and the remaining ones are fixed. 
Let $\sigma^j_i \in H^0(X, \mathcal{O}_X(N^j_i))$ be the defining 
section of $D^j_i$, for $j=1,\ldots, m, \, i=n_j+1,\ldots, \ell_j$.

\begin{cor} \label{C: coneeq}
	Let $Y_\bullet$ be an admissible flag on a Mori dream space $X$ such that
	 the conclusion of Proposition \ref{P:restriction}
	holds for  any SQM $f:X\dasharrow X'$ and any nef divisor $D'\subset X'$.
	Let furthermore  $D_1,\dots,D_r$ those  generators of a Mori chamber 
	$\Sigma$ which are movable. Then we have the identity
	\begin{align}
	&C(S(D_1,\dots,D_r)) \label{E: conepreim}\\ = 
	&q^{-1}(C(S_1(D_1,\dots,D_r))) \cap (\{0\} \times \R^{n-1}_{\geq 0} 
	\times C(D_1,\ldots, D_r)).  \nonumber
	\end{align}
\end{cor}
\begin{proof}
	This follows from the condition together with  the fact that
	there exists a SQM $\pi: X\dasharrow X'$  such that each divisor in 
	the cone $C(D_1,\ldots, D_r)$ is a pullback by $\pi$  of a nef divisor
	on $X'$ (\cite[Proposition 1.11(3)]{hk}). 
\end{proof}

We can now prove the following theorem which will---together with 
identity (\ref{E: conepreim})---enable us to inductively infer information on the shape of
global Okounkov bodies of certain Mori dream spaces.

\begin{thm} \label{T: indok}
Suppose in the above situation that each of the cones  \linebreak
$C(S(D^j_1,\ldots, D^j_{n_j}))$ is rational
polyhedral with generators given by vectors $w^j_1,\ldots, w^j_{r_j}$. Then 
the global Okounkov body 
$\Delta_{Y_\bullet}(X)$  is the cone 
generated by the vectors 
\begin{align}
(\nu(s_{Y_1}, [Y_1])), (\nu(\sigma^j_i), [D^j_i]), w^j_h,
\label{E: genok}
\end{align}
for
$$
	j=1,\ldots, m, \quad i=n_j+1,\ldots, \ell_j, \quad 
h=1,\ldots, r_j.
$$
\end{thm}

\begin{proof}

Let $E$ be an effective integral divisor on $X$, and let $s \in H^0(X, \mathcal{O}_X(E))$ be 
a nonzero section of $E$. Let $\nu_1(s)=a$. Then, $\zeta:=s/s_{Y_1}^a$, where $s_{Y_1} \in H^0(X, \mathcal{O}_X(Y_1))$ 
is the defining section of $Y_1$, is a section of $\mathcal{O}_X(E-aY_1)$ 
which vanishes to order $0$ along $Y_1$. Now, let $\Sigma=\mbox{conv}\{D_1,\ldots, D_\ell\}$
be a Mori chamber such that $E-aY_1 \in \Sigma$, and with generators ordered so that
 $D_1,\ldots, D_r$ are the movable generators. Let $E-aY_1=P+N$ be the corresponding decomposition of 
 $E-aY_1$ into its movable part $P$ and fixed part $N$.
Choose $M \in \N$ large enough such that all the divisors
$MP=c_1D_1+\cdots +c_{\ell} D_{r}$, $MN=c_{r+1}D_{r+1}+\cdots+c_\ell D_\ell$, where 
$c_1,\ldots, c_\ell \in \N_0$, and all $c_i D_i$ are integral divisors.
 Let $\sigma_i \in H^0(X, \mathcal{O}_X(D_i))$ be the 
defining section of $N_i, \, i=r+1,\ldots, \ell$.
The section $\zeta^M \in H^0(X, \mathcal{O}_X(m(E-aY_1)))$ now decomposes 
uniquely as a product
\begin{align*}
\zeta^M=\eta \sigma,
\end{align*}
where $\eta \in H^0(X, \mathcal{O}_X(c_1D_1+\cdots+c_r D_r))$, 
and $\sigma=\sigma_{r+1}^{c_{r+1}}\cdots\sigma_\ell^{c_\ell}$.
Since $\nu_1(\zeta)=0$, we also have $\nu_1(\eta)=0$.
Now, by assumption we have integral generators $w_1,\ldots, 
w_k \in \R^{n} \times \overline{\mbox{Eff}}(X)$
for the cone $C(S(D_1,\ldots, D_r))$, so that 
$(\nu(\eta), MP)=s_1w_1+\cdots+s_kw_k$, for some $s_1,\ldots, s_k \geq 0$.
Hence,   
\begin{align*}
(\nu(s), E)&=a(\nu(s_{Y_1}), Y_1)+\frac{c_{r+1}}{M}(v(\sigma_{r+1}), [D_{r+1}])+\cdots 
+\frac{c_\ell}{M}(v(\sigma_\ell), [D_\ell])\\
&+\frac{s_1}{M}w_1+\cdots+\frac{s_k}{M}w_k.
\end{align*}
It follows that $\Delta_{Y_\bullet}(X)$ lies in the closed convex cone generated by the 
vectors \eqref{E: genok}. Since all these vectors clearly belong to $\Delta_{Y_\bullet}(X)$, 
this finishes the proof.
\end{proof}

We are now in the position to prove the main result of this paper. 
Let us first define what we mean by a good flag on a Mori dream space.
\begin{defin}
	Let $X$ be a Mori dream space. An admissible flag $X= Y_0\supseteq Y_1\supseteq\dots\supseteq Y_n$ is \emph{good}, if
    \begin{enumerate}
        \item 
            $Y_i$ is a Mori dream space for each $0\le i\le n$, and for $i=1,\ldots, n$, $Y_i$ is 
	cut out by a global section $s_i\in H^0(Y_{i-1},\O_{Y_{i-1}}(Y_i))$, and
        \item
            for any small $\Q$-factorial modification 
            $f:Y_i\dasharrow Y_i'$, the exceptional locus $\exc(f)$
            intersects $Y_{i+1}$ properly.
    \end{enumerate}

\end{defin} 

Our main theorem now follows from the above results.
\begin{thm}\label{T: okgood}
	Assume that the  Mori dream space $X$ admits a good flag $Y_\bullet$. Then
	$
		\Delta_{Y_\bullet}(X)
	$
	is rational polyhedral.
\end{thm}
\begin{proof}
	We prove the theorem by induction over $n$. Every Mori dream curve
	is isomorphic to $\mP^1$, which for any choice of flag (i.e., choice of a point) has rational 
	polyhedral global Okounkov body, namely the cone in $\R^2$ spanned by the points
	$(0,1)$ and $(1,1)$.
	
	For the inductive step assume that $\Delta_{Y^1_\bullet}(Y_{1})$ is rational 
	polyhedral. By Theorem \ref{T: indok}, what we need to prove is that for any Mori chamber
	$\Sigma$ in $\Eff(X)$ the set of movable generators $D_1,\dots,D_r$ of $\Sigma$ 
	yield a rational polyhedral cone $C(S(D_1,\dots,D_r))$. 
	
	Since $Y_\bullet$ is a good flag, in particular 
	the assumptions of Proposition \ref{P:restriction} are satisfied for $Y_1\subset X$,
	so we can apply Corollary \ref{C: coneeq}. Since the linear map $q$ 
	(cf. \eqref{E: linmapq}) is 
	defined over $\Z$, equality (\ref{E: conepreim}) implies that $C(S(D_1,\dots,D_r))$ 
	is rational polyhedral if $C(S_1(D_1,\dots,D_r))$ is. Now the rational polyhedrality of 
	$C(S_1(D_1,\ldots, D_r))$ follows 
	from the rational polyhedrality of $\Delta_{Y^1_\bullet}(Y_{1})$ since
	$$
		C(S_1(D_1,\dots,D_r)) = 
		pr_2^{-1}(\Gamma(D_1 \cdot Y_1,\dots,D_r \cdot Y_1) \cap \Delta_{Y^1_\bullet}(Y_{1}),
	$$
	where $\Gamma(D_1 \cdot Y_1,\cdots, D_r \cdot Y_1) \subseteq \overline{\mbox{Eff}}(Y_1)$ 
	is the convex cone generated by the numerical equivalence classes of the divisors 
	$D_1 \cdot Y_1,\ldots, D_r \cdot Y_1$ on $Y_1$.
	
\end{proof}

\begin{rem}
It should be noted that the above result does not hold for general 
admissible flags of subvarieties of a Mori dream space. Indeed, \cite[Example 3.4]{KLM12} shows that 
$X:=\mathbb{P}^2 \times \mathbb{P}^2$ can be equipped with an admissible flag $Y_\bullet$ such that
the Okounkov body $\Delta_{Y_\bullet}(D)$, where $D$ is a divisor in the linear series 
$\mathcal{O}_{\mathbb{P}^2 \times \mathbb{P}^2}(3,1)$,
is not polyhedral. The flag $Y_\bullet$ here is of course not  a 
flag of Mori dream spaces: if $Y_1:=\mathbb{P}^2 \times E$, where $E \subseteq \mathbb{P}^2$ is a general elliptic curve, 
were a Mori dream space, then its image $E$ under the second projection would also be a Mori dream space 
by \cite[Theorem 1.1]{Oka11}. However, $\mathbb{P}^1$ is the only Mori dream curve (cf. \cite[p. 6]{C}).

Finally, \cite[Prop. 3.5]{KLM12} gives another example of a non-polyhedral Okounkov 
body of a divisor an a Mori dream space $Z$ with respect to a family of admissible flags $Y_\bullet$. 
However, the description of the pseudo-effective cone $\overline{\mbox{Eff}}(Y_1)$ shows that this 
cone is defined by quadratic equations and is thus not polyhedral; hence the divisor $Y_1$ on 
$Z$ is not a Mori dream space.
\end{rem}


\section{Okounkov bodies on Bott-Samelson varieties}\label{s: appl}

In this section we apply the general results from the previous section 
to Bott-Samelson varieties.

\subsection{Global Okounkov bodies}

On a Bott-Samelson variety $Z_w$ defined by a reduced sequence $w$ let 
$Y_\bullet$ be the natural  vertical flag described in section \ref{ss:VerticalFlag}.
In order to apply Theorem \ref{T: okgood} we will show that $Y_\bullet$ is 
in fact a good flag. The fact that each $Y_i$ is a Mori dream space follows from
Theorem \ref{T: BSMDS} and the second condition is a consequence of the 
following proposition together with the fact that the exceptional locus of a SQM equals
the stable base locus of the defining movable divisor.

\begin{prop}\label{P:BSrestr}
	Let $D$ be a Cartier divisor on $Z_\omega$, and $F$ a base component of 
	the complete linear series $|D|$. Then $F$ intersects $Y:=Y_1$ properly, i.e.,
	$$
		\rm{codim}_Y(F|_Y) = \rm{codim}_X(F)
	$$
		
\end{prop}

\begin{proof}
	We first prove that the base locus of $|D|$ is invariant under the group $P_1$.
	Let therefore $s\in H^0(X,\mathcal O_X(D))$ be a section,  and let	$p\in P_1$. 
	Then for any $x\in Z_\omega$
	$$
		s(px)= pp^{-1} s(px) = p( p^{-1}s(px)).
	$$
	By (\ref{E: repsections}), the right hand side 	
	is exactly $p((p^{-1}s)(x))$ and $(p^{-1}s\in H^0X,\mathcal O_X(D))$
	vanishes in $x$ if $x$ is in the base locus $B(D)$. Therefore,
	$$
	s(px) = p 0_{L_x} = 0_{L_{px}},
	$$
	and the claim follows. Then, every irreducible component of $B(D)$ is 
also $P_1$-invariant. In particular, $F$ is invariant under $P_1$. \bigskip

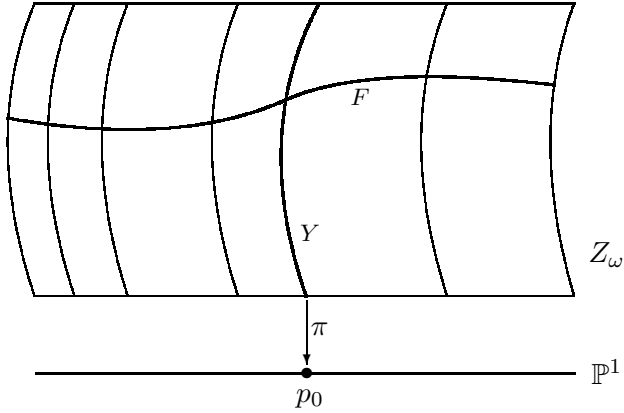
\unitlength 1mm 
\linethickness{0.4pt}
\ifx\plotpoint\undefined\newsavebox{\plotpoint}\fi 
\begin{picture}(79.57,60.545)(0,0)
\put(4.835,11.247){\line(1,0){71}}
\put(40.573,11.247){\circle*{1.33}}
\qbezier(75.7,60.3)(69.479,43)(75.7,21.5)
\qbezier(59,60.3)(52.188,43)(59,21.5)
\qbezier(31.5,60.3)(24.649,43)(31.5,21.5)
\qbezier(17,60.3)(10.196,43)(17,21.5)
\qbezier(10,60.3)(2.943,43)(10,21.5)
\put(75.891,60.3){\line(-1,0){71.1609}}
\qbezier(4.73,60.3)(-2.312,43)(4.73,21.5)
\put(4.73,21.5){\line(1,0){71}}
\thicklines
\qbezier(73.053,49.508)(48.352,52.451)(37.525,47.406)
\qbezier(37.525,47.406)(23.493,41.204)(1.261,45.093)
\qbezier(42.15,60.3)(33.11,43)(40.468,21.5)
\put(80,11.457){\makebox(0,0)[cc]{$\mP^1$}}

\put(41,8){\makebox(0,0)[cc]{$p_0$}}

\thinlines
\put(40.573,20.917){\vector(0,-1){8.514}}
\put(42.255,17.133){\makebox(0,0)[cc]{$\pi$}}
\put(80,27){\makebox(0,0)[cc]{$Z_\omega$}}
{\footnotesize
\put(41,29.852){\makebox(0,0)[cc]{$Y$}}
\put(47,48){\makebox(0,0)[cc]{$\ F$}}
}
\end{picture}
	
	Now, the $P_1$-action on $\mP^1$ has only one orbit, so $P_1$ 
	operates transitively on the fibres of $\pi$. Hence, the base component 
	$F$ is given by the union of orbits of elements in the restriction $F|_Y$. 
	Therefore, the generic fibre dimension holds for all fibres of $\pi$, so that
	$$
		\dim F = \dim F|_Y + 1, 
	$$
	which implies the statement.
	
\end{proof}

\begin{thm} \label{T: okbs}
Let $Z_w$ be a Bott-Samelson variety defined by a reduced sequence $w$. Then the 
global Okounkov body $\Delta_{Y_\bullet}$ is rational polyhedral.
\end{thm}

\begin{proof}
By Theorem \ref{T: BSMDS}, the variety $Z_w$ is a Mori dream space. The same 
is true for each $Y_i$. Moreover, by construction, $Y_i$ defines a Cartier divisor on $Y_{i-1}$ 
and is cut out by a global section $s_i$ which is just the pullback of a 
section $t_i\in H^0(\mP^1,\O_{\mP^1}(1))$ vanishing in the $B$-fixed point $p_0$. 
Inductive application of Proposition \ref{P:BSrestr} then shows that $Y_\bullet$ is 
a good flag and the result follows from Theorem \ref{T: okgood}.
\end{proof}

We can now also show that the Okounkov bodies of effective line 
bundles over Schubert varieties, with respect to a natural valuation-like 
function, are rational polyhedral. 
Indeed, Schubert varieties have rational singularities, so that 
the projection morphism $p_w: Z_w \to Y_{\overline{w}}$ 
satisfies the property $(p_w)_*\mathcal{O}_{Z_w}=\mathcal{O}_{Y_{\overline{w}}}$ 
(cf. \cite[Section 2.2]{B}). Hence, for any effective line bundle 
$L$ on $Y_{\overline{w}}$, we have 
\begin{align}
H^0(Y_{\overline{w}}, L) \cong H^0(Z_w, p_w^*L).
\label{E: pbschubert}
\end{align}
Let now $$\nu: \mbox{Cox}(Z_w)_h \setminus \{0\} \to \N_0^n,$$ 
where $\mbox{Cox}(Z_w)_h$ denotes the set of homogeneous elements 
in the Cox ring of $Z_w$ with respect to the effective basis, 
be the valuation-like function 
defined by the flag $Y_\bullet$, 
and let 
\begin{align*}
\nu_L: \bigsqcup_{k \geq 0} H^0(Y_{\overline{w}}, L^k) \setminus \{0\} 
\to \N_0^n
\end{align*}
be the valuation-like function naturally defined by the isomorphisms
\eqref{E: pbschubert} (for all powers $L^k$) and restriction of $\nu$. 
Then, the Okounkov body 
$\Delta_{\nu_L}(L)$ coincides with the slice 
$p_2^{-1}(p_w^*L) \cap \Delta_{Y_\bullet}(X)$ of $\Delta_{Y_\bullet}(X)$, 
and hence is rational polyhedral. Thus, we have proved the following corollary.
\begin{cor}
Let $L$ be an effective line bundle over the Schubert variety 
$Y_{\overline{w}}$ of $G/B$. Then, the Okounkov body 
$\Delta_{Y_\bullet}(L)$ defined by the natural valuation-like function 
$v_L$ defined by the flag $Y_\bullet$ in $Z_w$ is a rational polytope.
\end{cor}

If $Y_{\overline{w}}=G/B$ is a flag variety, the Picard group 
$\mbox{Pic}(G/B)$ has an effective basis, namely the line bundles 
$L_i=G \times_{\omega_i} \C$ defined by the fundamental weights 
$\omega_i, \, i=1,\ldots, r,$ with respect to a choice of simple roots 
for the root system of $G$. 
Let $\Sigma \subseteq \overline{\mbox{Eff}}(Z_w)$ be the closed 
convex cone generated by the divisors of the line bundles 
$p_w^*L_i, \, i=1,\ldots, r$. 
By the isomorphisms \eqref{E: pbschubert} we now have 
\begin{align*}
\Delta_{Y_\bullet}(G/B) \cong p_2^{-1}(\Sigma) \cap \Delta_{Y_\bullet}(Z_w).
\end{align*}
Since the cone $\Sigma$ is finitely generated, the cone on the right hand side is 
rational polyhedral, so that we have proved the following corollary.

\begin{cor}
The global Okounkov body $\Delta_{Y_\bullet}(G/B)$ of the flag variety $G/B$, 
with respect to the valuation defined by the flag $Y_\bullet$ of subvarieties of 
$Z_w$, is a rational polyhedral cone.
\end{cor}


\subsection{Weight multiplicities}
We now turn our attention to the action of 
a torus $H \subseteq B$, contained in a maximal torus of $G$ lying in $B, $ on the section ring 
$R(D):=\bigoplus_{k \geq 0} H^0(Z_w, \mathcal{O}_{Z_w}(kD))$ of an effective 
divisor $D$ on $Z_w$. Recall that each section 
space $H^0(Z_w, \mathcal{O}_{Z_w}(kD))$ carries a representation 
of $B$ given by the action of $B$ as automorphisms of the 
line bundle $\mathcal{O}_{Z_w}(kD)$ (cf. \cite{LT04}). 
Moreover, the flag $Y_\bullet$ consists of 
$B$-invariant subvarieties of $Z_w$, so that 
the valuation-like function 
\begin{align*}
 \nu_{D}: \bigsqcup_{k \geq 0} H^0(Z_w, \mathcal{O}_{Z_w}(kD)) \setminus \{0\} \to \N_0^n
\end{align*}
is $B$-invariant, i.e., the identity $\nu(b.s)=\nu(s)$ holds for any non-zero 
section $s \in H^0(Z_w, \mathcal{O}_{Z_w}(kD))$, and $b \in B$. 
Hence, there is a well-defined projection 
\begin{align*}
q: \Delta_{Y_\bullet}(D) \to \Pi_D
\end{align*}
onto the weight polytope (cf. \cite{Br86}) of the section ring $R(D)$ for the action of 
the torus $H$ (cf. \cite{Ok96} \cite{KK10}). If $\mathfrak{h}=\mbox{Lie}(H)$ is the Lie algebra of $H$, 
and the $\mu \in \Pi_D \subseteq \mathfrak{h}$ is a rational point 
in the interior of the weight polytope, we then 
have that the asymptotics of the weight spaces 
$W_{k\mu} \subseteq H^0(Z_w, \mathcal{O}_{Z_w}(kD))$ are 
given by
\begin{align*}
\lim_{k \rightarrow \infty} \frac{\mbox{dim} W_{k\mu}}{k^{d-r}}=\mbox{vol}_{d-r}(q^{-1}(\mu) 
\cap \Delta_{Y_\bullet}(D)),
\end{align*}
where $r$ is the dimension of the moment polytope $\Pi_D$, $d$ is the dimension 
of the Okounkov body $\Delta_{Y_\bullet}(D)$ (and which equals the 
Iitaka dimension of the line bundle $\mathcal{O}_{Z_w}(D)$), and the 
right hand side denotes the $(d-r)$-dimensional Lebesgue measure of the slice 
$q^{-1}(\lambda) \cap \Delta_{Y_\bullet}(D)$ of the 
Okounkov body $\Delta_{Y_\bullet}(D)$. We thus get the following result, saying the 
the asymptotics of weight spaces are given by polyhedral expressions.

\begin{cor}
For any effective divisor $D$ on $Z_w$, and rational point $\mu \in \Pi_D$ in the 
interior of $\Pi_D$, the asymptotic 
multiplicity $$\lim_{k \rightarrow \infty} \frac{\mbox{dim} W_{k\mu}}{k^{d-r}}$$ is 
the volume of a rational polytope. As a consequence, the same 
holds for the weight spaces $$W_{k\mu} \subseteq H^0(Y_{\overline{w}}, L^k)$$ for 
an effective line bundle $L$ over a Schubert variety $Y_{\overline{w}}$.
\end{cor}

\begin{proof}
We only need to prove the second claim about Schubert varieties. 
Here we notice that the projection morphism 
$p_w: Z_w \to Y_{\overline{w}}$ is $B$-equivariant, and so in 
particular $H$-invariant. Hence, the isomorphisms 
\eqref{E: pbschubert} for the powers $L^k$ are $H$-equivariant, 
so that the claim thus follows from the first part about Bott-Samelson 
varieties.
\end{proof}

\end{document}